\documentclass[11pt,a4paper]{amsart}

\title[Small values of signed harmonic sums]{Small values of signed harmonic sums}

\author[S.~Bettin]{Sandro Bettin}
\address{
  Dipartimento di Matematica\\
  Universit\`{a} di Genova\\
  Via Dodecaneso 35\\
  16146 Genova,~Italy}
\email{bettin@dima.unige.it}

\author[G.~Molteni]{Giuseppe Molteni}
\address{
  Dipartimento di Matematica\\
  Universit\`{a} di Milano\\
  Via Saldini 50\\
  20133 Milano,~Italy}
\email{giuseppe.molteni1@unimi.it}

\author[C.~Sanna]{Carlo Sanna}
\address{
  Dipartimento di Matematica\\
  Universit\`{a} di Torino\\
  Via Carlo Alberto 10\\
  10123 Torino,~Italy}
\email{carlo.sanna.dev@gmail.com}

\usepackage{amsmath}
\usepackage{amssymb}
\usepackage{amsthm}
\usepackage{geometry}
\usepackage{graphicx}
\usepackage{color}
\usepackage{hyperref}
\hypersetup{colorlinks=true}
\usepackage[normalem]{ulem}

\newtheorem{thm}{Theorem}[section]
\newtheorem{lem}[thm]{Lemma}
\newtheorem{pro}[thm]{Proposition}
\newtheorem{corol}[thm]{Corollary}
\theoremstyle{definition}
\newtheorem{remark}[thm]{Remark}
\newtheorem*{acknowledgements}{Acknowledgements}

\newcommand{\N}{\mathbb{N}}
\newcommand{\Z}{\mathbb{Z}}
\newcommand{\R}{\mathbb{R}}
\newcommand{\m}{\mathfrak{m}}
\newcommand{\s}{\mathfrak{s}}
\renewcommand{\SS}{\mathfrak{S}}
\newcommand{\eps}{\varepsilon}
\newcommand{\np}{64}
\DeclareMathOperator{\lcm}{lcm}
\newcommand{\intpart}[1]{\left\lfloor#1\right\rfloor}

\numberwithin{equation}{section}
\begin{document}
\begin{abstract}
For every $\tau\in\R$ and every integer $N$, let $\m_N(\tau)$ be the minimum of the distance of $\tau$
from the sums $\sum_{n=1}^N s_n / n$, where $s_1, \ldots, s_n \in \{-1, +1\}$. We prove that $\m_N(\tau)
< \exp\!\big(-C(\log N)^2\big)$, for all sufficiently large positive integers $N$ (depending on $C$ and
$\tau$), where $C$ is any positive constant less than $1/\log 4$.
\end{abstract}

\maketitle

\begin{center}
To appear in C. R. Math. Acad. Sci. Paris 2018.\\
\end{center}

\section{Introduction}
For each positive integer $n$, let
\begin{equation*}
H_n := 1 + \frac{1}{2} + \frac{1}{3} + \cdots + \frac{1}{n}
\end{equation*}
be the $n$th \emph{harmonic number}. Harmonic numbers have long been an active area of research. For
instance, Wolstenholme~\cite{Wol62} proved that for any prime number $p \geq 5$ the numerator of $H_{p -
1}$ is divisible by $p^2$; while Taeisinger~\cite[p.~3115]{Wei02} showed that $H_n$ is never an integer
for $n > 1$. This latter result has been generalized by Erd\H{o}s~\cite{Erd32} to sums of inverses of
numbers in arithmetic progression. Also, the $p$-adic valuation of $H_n$ has been studied by
Boyd~\cite{MR1341721}, Eswarathasan and Levine~\cite{MR1129989}, Wu and Chen~\cite{MR3608179}, and
Sanna~\cite{MR3486261}. Moreover, harmonic numbers are special cases of \emph{Egyptian fractions}
(rational numbers which are sums of distinct unit fractions), themselves an active area of
research~\cite[\S{D11}]{MR2076335}.

It is well known that $H_n \to +\infty$ as $n \to +\infty$. More precisely,
\begin{equation}\label{equ:harmonic}
H_n = \log n + \gamma + O\!\left(1/{n}\right)
\end{equation}
for all positive integers $n$, where $\gamma$ is the \emph{Euler--Mascheroni constant}.
On the other hand, the alternating signs harmonic number
\begin{equation*}
H_n^\prime = 1 - \frac{1}{2} + \frac{1}{3} - \cdots + \frac{(-1)^{n + 1}}{n}
\end{equation*}
converges to $\log 2$ as $n \to +\infty$. Building on earlier work by Morrison~\cite{Morrison2,Mor},
Schmuland~\cite{MR2040884} proved that the random harmonic series
\begin{equation*}
X:=\sum_{n \,=\, 1}^\infty \frac{s_n}{n} ,
\end{equation*}
where $s_1, s_2, \ldots$ are independent uniformly distributed random variables in $\{-1,+1\}$, converges
almost surely to a random variable with smooth density function $g$ supported on the whole real line.
Interestingly, $g(0)$ and $g(2)$ are extremely close to, but slightly smaller than, $\frac{1}{4}$ and
$\frac{1}{8}$ respectively (the error being of the order of $10^{-6}$ and $10^{-43}$ respectively).
We refer to~\cite[p.~101]{MR2051473} and~\cite{MR2040884} for some more information on these constants and
to~\cite{Crandall,Offner} for more information on the random variable $X$.

\begin{figure}[h]
\centering
\includegraphics[width=0.50\textwidth]{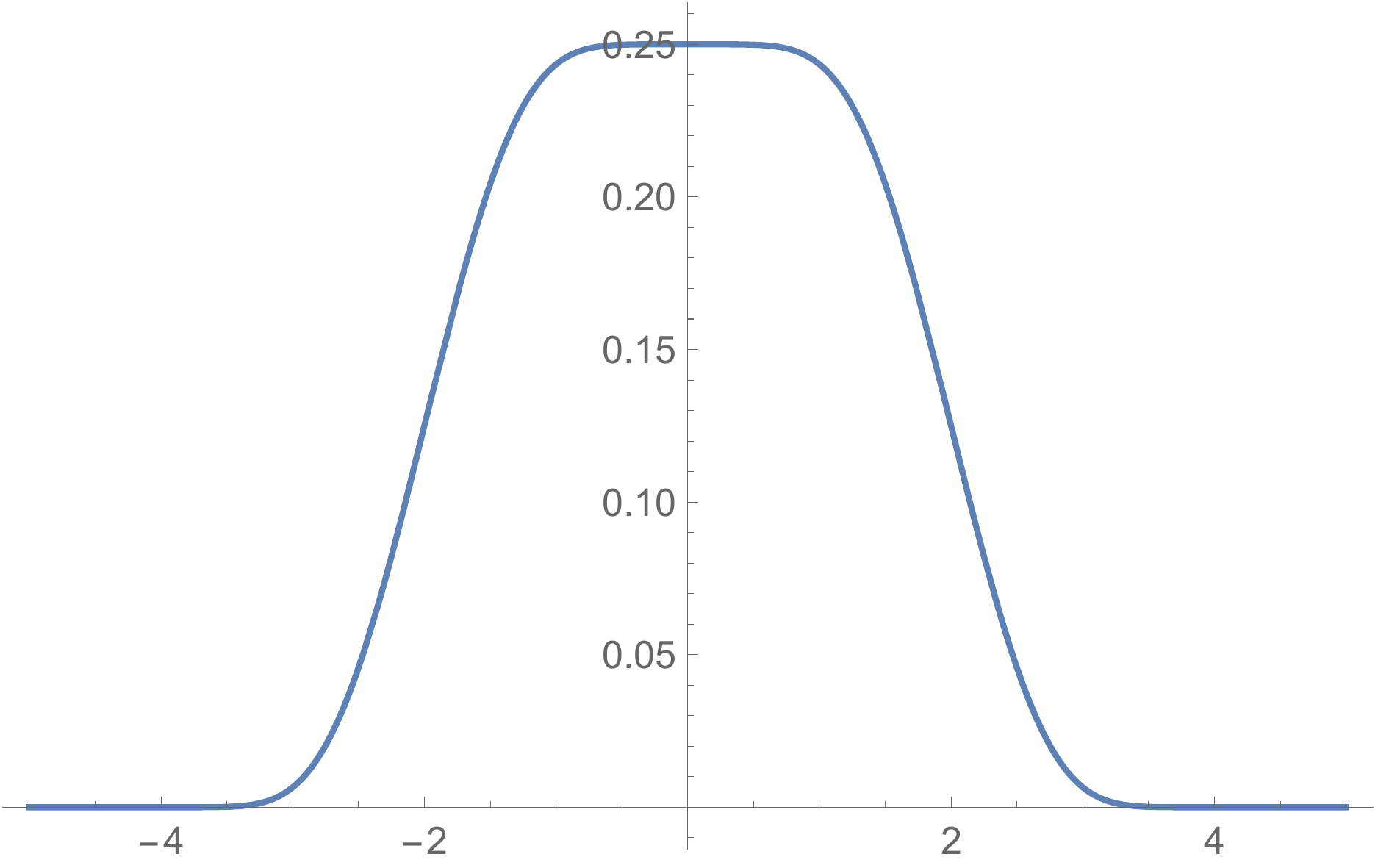}
\caption{The graph of the distribution function $g(x)$ of $X$.}
\label{Fig:1}
\end{figure}

In this paper we are interested in the set
\begin{equation*}
\SS_N := \left\{\sum_{n \,=\, 1}^N \frac{s_n}{n} : s_1, \ldots, s_N \in \{-1,+1\}\right\} .
\end{equation*}
Clearly, $\SS_N$ is symmetric respect to the origin and
\begin{equation*}
\max \SS_N = H_N \sim \log N
\end{equation*}
as $N \to +\infty$, by~\eqref{equ:harmonic}.
On the other hand, the quantity
\begin{equation*}
\m_N := \min\left\{|\s| : \s \in \SS_N\right\}
\end{equation*}
is much more mysterious. It is not difficult to prove (see Proposition~\ref{pro:nonzero} below) that
$\m_N\neq 0$ for all $N\in\N$. In particular, estimating the least common multiple of the denominators
using the Prime Number Theorem, one easily obtains the following lower bound for $\m_N$,
\begin{equation}\label{equ:simpleintro}
\m_N > \exp\!\left(-N + o(N) \right),
\end{equation}
as $N \to +\infty$.

More generally, we shall study the function
\begin{equation*}
\m_N(\tau):= \min\left\{|\s-\tau|\colon \s \in \SS_N\right\}, \qquad\tau\in\R.
\end{equation*}
Using an easy argument, in Proposition~\ref{pro:lower} below we show that for almost every $\tau$,
\begin{equation}\label{equ:A3}
\m_N(\tau) > \exp\!\left(-0.665\,N + o(N) \right) ,
\end{equation}
as $N \to +\infty$ (notice that $0.665<\log 2= 0.693\ldots$). This bound holds for almost every $\tau$,
but not for all of them: in fact, $\m_N(\tau)$ can be arbitrary small infinitely often. Precisely, given
any $f \colon \N \to \R_{>0}$ we can construct $\tau_f \in \R$ such that $\m_N(\tau_f) < f(N)$ for
infinitely many $N$ (see~\cite[Proposition~5.9]{BettinMolteniSanna2}). The bound in~\eqref{equ:A3} is not
optimal, and some minor variations of our proof are already able to produce some small improvement.

In this paper we are mainly interested in the opposite direction where the upper bound for $\m_N(\tau)$
is sought. Our main result is the following.
\begin{thm}\label{thm:upper}
For every $\tau\in \R$ and for any positive constant $C$ less than $1/\log 4$, we have
\begin{equation}\label{equ:upbound}
\m_N(\tau) < \exp\!\left(-C (\log N)^2\right) ,
\end{equation}
for all sufficiently large $N$, depending on $C$ and $\tau$.
\end{thm}
Notice that a sequence of signs $s_{1},\dots,s_{N}$ realizing the minimum in the definition of
$\m_{N}(\tau)$ does not come from a ``universal'' infinite sequence $(s_n)_{n\geq1}$ such that, setting
$\sigma_N :=\sum_{n\,=\,1}^N s_n / n$, we have $\m_N(\tau) = |\sigma_N - \tau|$ for all $N$. Indeed,
$|\sigma_N-\sigma_{N-1}|=1/N$ and so $\sigma_N$ and $\sigma_{N-1}$ cannot both be less then $1/(2N)$ away
from $\tau$.

The upper and lower bounds given in the inequalities~\eqref{equ:simpleintro} and~\eqref{equ:upbound} are
quite distant and thus they do not indicate clearly what is the real size of $\m_N$. A heuristic argument
suggests that the inequality $\m_N > \exp\!\left(-\frac{1}{2}N + o(N)\right)$ is satisfied for infinitely
many $N$, and numerical computations (cf.~Figure~\ref{Fig:2}) might suggest that actually $\m_N =
\exp\!\left(-\frac{1}{2}N + o(N)\right)$. However, because of the exponential nature of the problem, we
were able to compute only the first $\np$ values of $\m_N$, which are clearly not enough to draw a solid
conclusion. We shall give these values of $\m_N$ in the appendix. Despite the limited amount on data at
disposal, some interesting observations can be drawn from them. For example, $\m_N$ is not a decreasing
function of $N$ and there are several repeated values. One can then perhaps expect that there are
infinitely many values of $N$ such that $\m_N=\m_{N+1}$ or even such that $\m_N=\cdots=\m_{N+k}$ for any
fixed $k\in\N$.

\begin{figure}[h]
\centering
\includegraphics[width=0.60\textwidth]{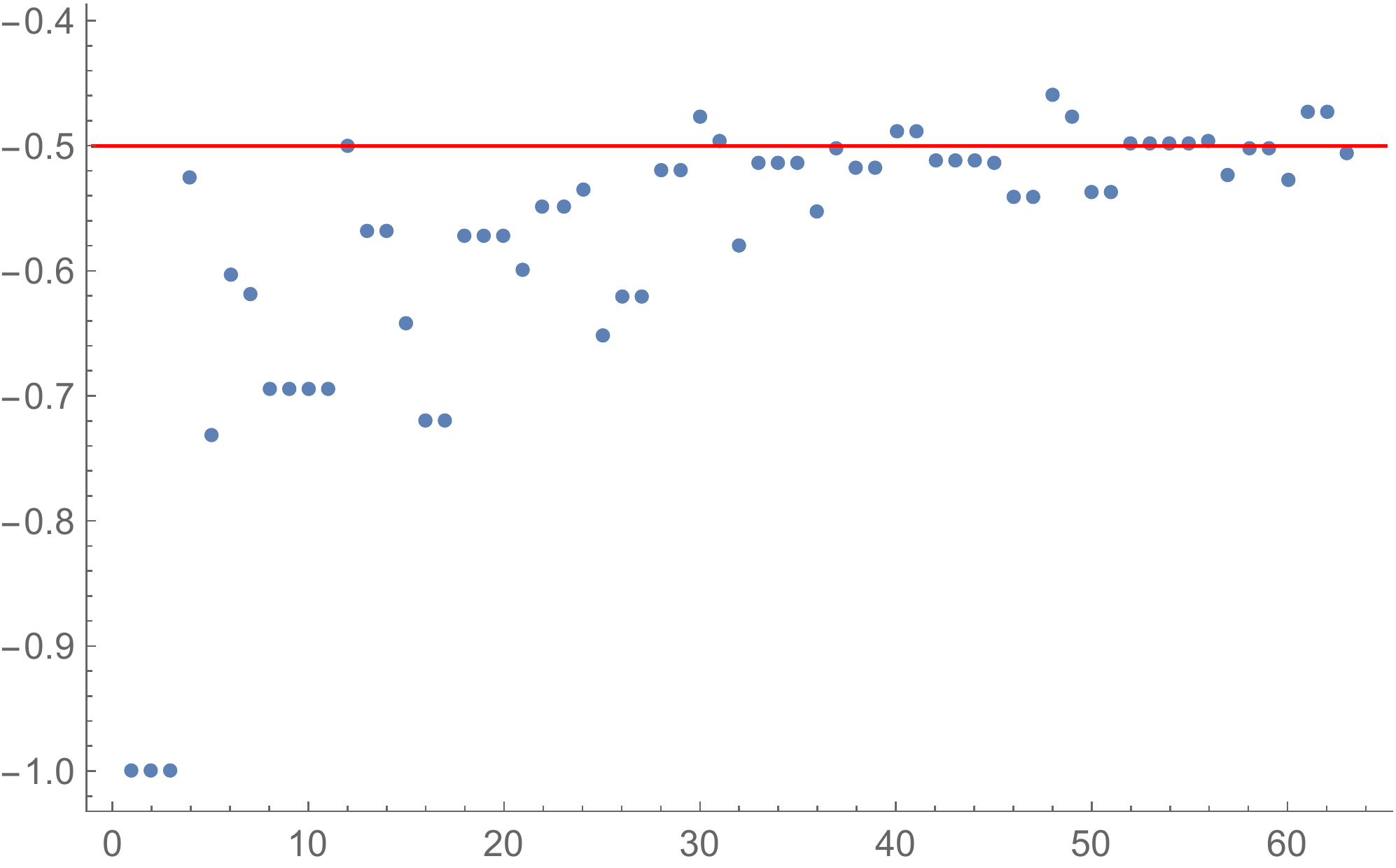}
\caption{The first $\np$ values of $\frac{\log \m_N}{\log L_N}$, where $L_N= \lcm\{1, \ldots,
N\}=e^{N+o(N)}$, plotted against the constant line $-\frac{1}{2}$.} \label{Fig:2}
\end{figure}

We prove Theorem~\ref{thm:upper} using a probabilistic argument. More precisely, in
Theorem~\ref{thm:loclimthm} below we shall prove a small scale distribution result for
$X_N:=\sum_{n\,=\,1}^N \frac{s_n}{n}$,  where $s_1,\dots,s_N$ are independently uniformly distributed
random variables in $\{-1, +1\}$. Theorem~\ref{thm:upper} will follow immediately from this result
(cf.~Corollary~\ref{cor:maincorollary}). Interestingly, this distribution problem for $X_N$ will lead us
to another classical number theoretic problem: that of bounding a short average of the number of divisors
in a prescribed small interval. We will attack this problem in two different ways, first using Rankin's
trick together with a bound for the divisor function $\sigma_s(n)$ proved in Ramanujan's lost
notebook~\cite{MR1606180}, and then using a more complicated arithmetic construction. Surprisingly, the
two methods both lead to the same bound~\eqref{equ:upbound}, albeit with different constants.

While the probabilistic approach has the advantage of showing the existence of several $N$-tuples of
signs $s_1, \ldots, s_N$ giving small values for $|\sigma_N - \tau|$, this approach does
not produce any explicit instance of these $N$-tuples. If one is interested in exhibiting explicit
sequences, then one can construct some special signed harmonic series converging to $\tau$ and estimate
the absolute value of their partial sums. A natural candidate is the ``greedy'' sequence obtained by
setting $s_{N+1} := +1$ if $\sigma_N \leq \tau$, and $s_{N+1} := -1$ otherwise. It is clear that $\sigma_N$
converges to $\tau$, since at each step one chooses the sign which makes $\sigma_N$ closer to $\tau$ and more
precisely one has $|\sigma_N-\tau| \leq 1/N$ for all $N$ large enough (depending on $\tau$). On the other
hand, as observed above, $\sigma_N$ cannot be always very close to $\tau$ and in fact the inequality
$|\sigma_N - \tau| \geq 1/(N+1)$ is satisfied infinitely often. However, it is still possible to prove that for any $A
> 0$ one has $|\sigma_N - \tau|\ll_A N^{-A}$ for infinitely many positive integers $N$.
In fact we can show that for almost all $\tau$ one has
\begin{equation*}
\liminf_{n \to +\infty} \frac{\log|\sigma_N - \tau|}{(\log N)^2} = -\frac1{\log 4} .
\end{equation*}
It is quite remarkable that this ``greedy'' algorithm and the probabilistic method developed in this
paper both give a decay rate of $\exp\!\big(-\big(\tfrac1{\log 4} + o(1)\big) (\log N)^2\big)$. The study
of this ``greedy'' sequences needs completely different tools from those employed here, thus we leave its
study to another paper~\cite{BettinMolteniSanna2}.

\begin{acknowledgements}
S.~Bettin is member of the INdAM group GNAMPA.
G.~Molteni and C.~Sanna are members of the INdAM group GNSAGA.
The work of the first and second author is partially supported by PRIN 2015 ``Number Theory and Arithmetic Geometry''.
The authors would also like to thank D.~Koukoulopoulos and M.~Radziwi\l\l{} for several useful discussions.
The authors would also like to thank the anonymous referee for carefully reading the paper.
\end{acknowledgements}

\subsection*{Notation}
We employ the Landau--Bachmann ``Big Oh'' and ``little oh'' notations $O$ and $o$, as well as the
associated Vinogradov symbols $\ll$ and $\gg$, with their usual meanings. Any dependence of the implied
constants is explicitly stated or indicated with subscripts. As usual, we write $\mathbb{E}[X]$ for the
expected valued of a random variable $X$, and $\mathbb{P}[E]$ for the probability of an event $E$. Also,
we indicate with $\mathcal{C}_c(\R)$ the space of continuous functions with compact support on $\R$ and
with $\mathcal{C}^\infty_c(\R)$ the subspace of  $\mathcal{C}_c(\R)$ consisting of smooth functions.
Finally, for each $\Phi \in \mathcal{C}_c(\R)$ we let $\widehat{\Phi}$ denote its Fourier transform, here
defined by
\begin{equation*}
\widehat{\Phi}(x) := \int_{\R} \Phi(y) e^{-2\pi ixy} \,\mathrm{d} y
\end{equation*}
for all $x \in \R$.

\section{The small scale distribution of $X_N$}
We start with stating our result on the small scale distribution of $X_N$. We remind that $X_N$ is the
random variable defined by $X_N:=\sum_{n\,=\,1}^N s_n / n$, where $s_n$ are taken uniformly and
independently at random in $\{-1,+1\}$.
\begin{thm}\label{thm:loclimthm}
Let $C$ be any positive constant less than $1/\log 4$.
Then, for all intervals $I\subseteq\R$ of length
$|I|>\exp\!\left(-C (\log N)^2\right)$ one has
\begin{equation*}
\mathbb{P}[X_N\in I] = \int_I g(x) \, \mathrm{d} x+o(|I|),
\end{equation*}
as $N\to\infty$, where
\begin{equation*}
g(x):=2\int_{0}^{+\infty}\cos(2\pi u x) \prod_{n=1}^{+\infty}\cos(2\pi u/n)\,\mathrm{d} u.
\end{equation*}
\end{thm}
\begin{remark}\label{rem:distribfcn}
As shown by Schmuland~\cite{MR2040884}, $g(x)$ is a smooth strictly positive function which is
$O_A\!\left(x^{-A}\right)$ as $x \to \pm\infty$, for any $A > 0$.
\end{remark}
\begin{corol}\label{cor:maincorollary}
Let $C$ be any positive constant less than $1/\log 4$.
Then, for all $\tau\in\R$ one has
\begin{equation*}
\#\left\{(s_1, \ldots, s_N )\in \{-1,+1\}^N : \left|\tau - \sum_{n \,=\, 1}^N \frac{s_n}{n} \right| < \delta \right\}
\sim  2^{N+1} g(\tau) \delta(1 +o_{C,\tau}(1))
\end{equation*}
as $N\to\infty$ and $\delta\to 0$, uniformly in $\delta \geq \exp\!\left(-C (\log N)^2\right)$.
In particular, for all large enough $N$ one has $\m_N (\tau)<\exp\!\left(-C (\log N)^2\right)$.
\end{corol}
\begin{proof}
The result follows immediately from Theorem~\ref{thm:loclimthm} and Remark~\ref{rem:distribfcn}.
\end{proof}

We now proceed to proving Theorem~\ref{thm:loclimthm}.
For each $N\in\N\cup\{\infty\}$ and for any real number $x$, define the product
\begin{equation*}
\varrho_N(x) := \prod_{n\,=\,1}^N \cos\!\Big(\frac{\pi x}{n}\Big)
\end{equation*}
and let $\varrho(x):=\varrho_\infty(x)$.
\begin{lem}\label{lem:expected}
We have
\begin{equation*}
\mathbb{E}[\Phi(X_N)] = \int_{\R} \widehat{\Phi}(x) \varrho_N(2x) \, \mathrm{d} x
\end{equation*}
for all $\Phi \in \mathcal{C}_c^1(\R)$.
\end{lem}
\begin{proof}
By the definition of expected value and by using inverse Fourier transform, we get
\begin{align*}
\mathbb{E}[\Phi(X_N)]
&= \frac{1}{2^N} \sum_{s_1,\ldots,s_N\,\in\,\{-1,+1\}} \Phi\!\left(\sum_{n\,=\,1}^N \frac{s_n}{n}\right) \\
&= \frac{1}{2^N} \sum_{s_1, \ldots, s_N \in \{-1, +1\}}
   \int_{\R} \widehat{\Phi}(x) \exp\!\left(2\pi ix \sum_{n\,=\,1}^N \frac{s_n}{n}\right) \mathrm{d} x \\
&= \frac{1}{2^N}
   \int_{\R} \widehat{\Phi}(x) \sum_{s_1, \ldots, s_N \,\in\, \{-1, +1\}} \exp\!\left(2\pi ix \sum_{n\,=\,1}^N \frac{s_n}{n}\right) \mathrm{d} x \\
&= \int_{\R} \widehat{\Phi}(x) \varrho_N(2x) \,\mathrm{d} x,
\end{align*}
as desired.
\end{proof}
In the following lemma, whose proof we postpone to Section~\ref{sec:rho}, we collect some results on $\varrho_N$.
\begin{lem}\label{lem:rhoN}
For all $N\in\N$ and $x\in\big[0,\sqrt{N}\big]$ we have
\begin{equation}\label{equ:rho2}
\varrho_N(x) =\varrho(x) (1+O(x^2/N)).
\end{equation}
Moreover, there exist absolute constants $B,C,E> 0$ such that
\begin{equation}\label{equ:rho3}
|\varrho_N(x)| < \exp\!\left(-B \exp\!\left(E\sqrt{\log x}\right)\right)
\end{equation}
for all sufficiently large positive integers $N$ and for all $x \in \left[1, \exp\!\left(C(\log N)^2\right)\right]$.
In particular, $C$ can be taken as any positive real number less than $1/\log 4$.
\end{lem}

\begin{figure}[h]
\centering
\includegraphics[width=0.60\textwidth]{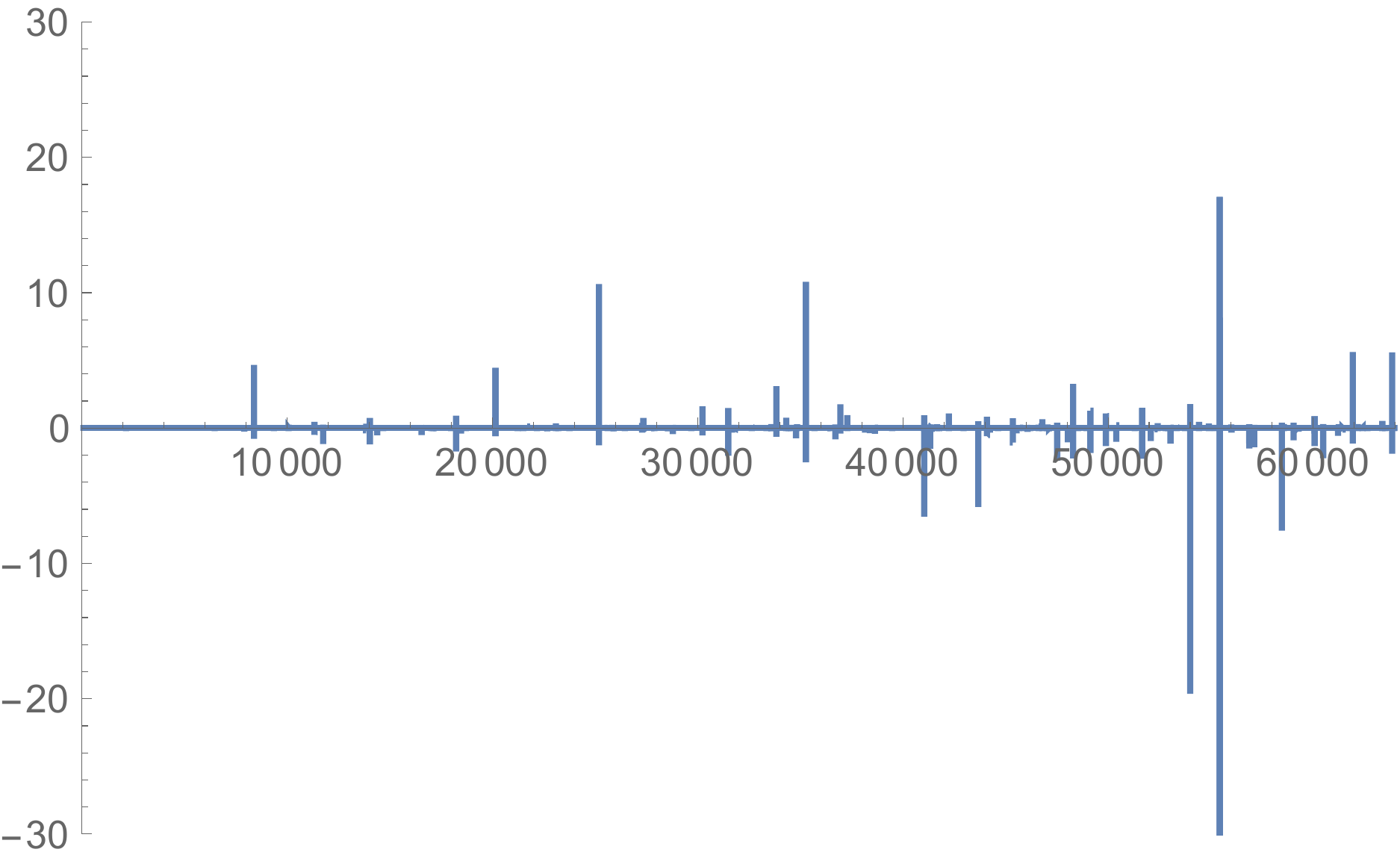}
\caption{Plot of $\varrho_{N}(x)/x$ for $N=40$ and $x\in[0,N^3]$.}
\label{Fig:3}
\end{figure}

We are now in a position to prove Theorem~\ref{thm:loclimthm}. Let $C$ be any positive constant less than
$1/\log 4$. Pick any small $\eps>0$, and set $\xi_{N,\pm \eps} := \exp\!\left(-(1\pm\eps)C(\log
N)^2\right)$ and $\xi_N:=\xi_{N,0}$. For an interval $I=[a,b]$ with $b-a>2\xi_N$, let $\Phi_{N,\eps}^+$
(respectively $\Phi_{N,\eps}^-$) be a smooth function $\R\to[0,1]$ with support contained in $[a -
\xi_{N,-\eps}, b + \xi_{N,-\eps}]$ (resp. $I$) and with $\Phi_{N,\eps}^+(x)=1$ if $x\in I$, (resp. $x \in
[a + \xi_{N,-\eps}, b - \xi_{N,-\eps}]$). Also, suppose $\Phi_{N,\eps}^{\pm(j)}(x)\ll_j
(\xi_{N,-\eps})^{-j}$ for all $j\geq 0$. It is not difficult to construct such functions. It follows that
the Fourier transform of $\Phi_{N,\eps}^{\pm}$ satisfies
\begin{equation}\label{equ:phibound}
\widehat{\Phi}_{N,\eps}^{\pm}\ll_A (1+|x|\xi_{N,-\eps})^{-A}
\end{equation}
for all $A>0$ and all $x\in\R$. Since
\begin{equation*}
\mathbb{E}[\Phi_{N,\eps}^{-}(X_N)] \leq \mathbb{P}[X_N\in I] \leq \mathbb{E}[\Phi_{N,\eps}^{+}(X_N)] ,
\end{equation*}
it suffices to show
\begin{equation*}
\mathbb{E}[\Phi_{N,\eps}^{\pm}(X_N)] = \int_\R \Phi_{N,\eps}^{\pm}(x) g(x) \, \mathrm{d} x+o_\eps(|I|) ,
\end{equation*}
because this quantity is evidently equal to $\int_I g(x) \, \mathrm{d} x + o(|I|)$.
From now on, let
$\Phi_{N,\eps}$ be one of $\Phi_{N,\eps}^{+}$, $\Phi_{N,\eps}^{-}$.
By Lemma~\ref{lem:expected} we have
\begin{equation*}
\mathbb{E}[\Phi_{N,\eps}(X_N)]
= \frac{1}{2}\int_{\R} \widehat{\Phi_{N,\eps}}(x/2) \varrho_N(x) \, \mathrm{d} x
= I_1+I_2+I_3,
\end{equation*}
where $I_1$, $I_2$, $I_3$ are the integral supported in $|x| < N^\eps$, $|x|\in[N^\eps, M_N^{1+\eps}]$,
and $|x|> M_N^{1+\eps}$, respectively, where $M_N := \xi_{N,-\epsilon}^{-1}$. Note that $M_N^{1+\eps}>
N^\eps$ for $N$ large enough, that $M_N^{1+\eps}=\xi_N^{-(1-\eps^2)}< \xi_N^{-1}$, and that
$M_N^{1+\eps}\xi_{N,-\eps} = \xi_{N,-\eps}^{-\eps} =\xi_{N}^{-\eps(1-\eps)}$ goes to infinity when $N$
goes to infinity.

Offner~\cite{Offner} showed that $\varrho(x)$ decays double exponentially. In particular, using
also~\eqref{equ:rho2}, we have
\begin{align*}
I_1
&=\frac{1}{2}\int_{-N^\eps}^{N^\eps} \widehat{\Phi_{N,\eps}}(x / 2) \varrho_N(x) \, \mathrm{d} x
 =\frac{1}{2}\int_{-N^\eps}^{N^\eps} \widehat{\Phi_{N,\eps}}(x / 2) \varrho(x)\, \mathrm{d} x
  + O\!\left(\|\widehat{\Phi_{N,\eps}}\|_\infty N^{-1+3\eps}\right) \\
&=\int_\R \widehat{\Phi_{N,\eps}}(x) \varrho(2x)\, \mathrm{d} x+O_\eps\!\left(\|{\Phi_{N,\eps}}\|_1 N^{-1+3\eps}\right).
\end{align*}
By~\eqref{equ:rho3} if $N$ is sufficiently large we have
\begin{align*}
|I_2|
&\leq \|\widehat{\Phi_{N,\eps}}\|_\infty\int_{N^\eps}^{M_N^{1+\eps}} |\varrho_N(x)| \,\mathrm{d} x \\
&\leq \|\Phi_{N,\eps}\|_1 \int_{N^\eps}^{+\infty} \!\!\exp\!\left(-B\exp\!\left(E\sqrt{\log x}\right)\right)\,\mathrm{d} x
\ll_\eps \|\Phi_{N,\eps}\|_1 N^{-1}.
\end{align*}
Now, by~\eqref{equ:phibound} we easily have
\begin{align*}
|I_3|
&\leq \int_{|x| > M_N^{1+\eps}} |\widehat{\Phi_{N,\eps}}(x)| \,\mathrm{d} x
 \ll_A \int_{M_N^{1+\eps}}^{+\infty} (1+ x\xi_{N,-\eps})^{-A} \,\mathrm{d} x\\
&\ll_A (1+ M_N^{1+\eps}\xi_{N,-\eps})^{1-A}
\ll_A \xi_{N}^{\eps(1-\eps)(A-1)}
= o_\eps(\xi_N)
= o_\eps(|I|) ,
\end{align*}
where in the last steps we have chosen $A=1+2/\eps$.
Thus, collecting the above results
\begin{align*}
\mathbb{E}[\Phi_{N,\eps}(X_N)]
&= \int_{\R} \widehat{\Phi_{N,\eps}}(x ) \varrho(2x) \, \mathrm{d} x+O_\eps(\|{\Phi_{N,\eps}}\|_1 N^{-1+3\eps})+ o_\eps(|I|)\\
&= \int_{\R} \Phi_{N,\eps}(x) g(x) \, \mathrm{d} x+O_\eps(\|{\Phi_{N,\eps}}\|_1 N^{-1+3\eps}) + o_\eps(|I|),
\end{align*}
by Parseval's theorem and the proof of Theorem~\ref{thm:loclimthm} is completed, because
$\|{\Phi_{N,\eps}}\|_1 = O_\eps(|I|)$.
\smallskip

We conclude the section with the following propositions which prove the bounds~\eqref{equ:simpleintro}
and~\eqref{equ:A3}.
\begin{pro}\label{pro:nonzero}
We have $\m_N \neq 0$ for each positive integer $N$. Moreover, as $N\to\infty$,
\begin{equation}\label{equ:simple}
\m_N > \exp\!\left(-N + o(N)\right).
\end{equation}
\end{pro}
\begin{proof}
For each positive integer $N$, define $L_N := \lcm\{1, \ldots, N\}$. Let $k$ be the unique nonnegative
integer such that $2^k \leq N < 2^{k+1}$. Then, for all $n \in \{1, \ldots, N\}$, we have that $L_N / n$
is an integer which is odd if and only if $n = 2^k$. As a consequence, for all $s_1, \ldots, s_N \in
\{-1, +1\}$, we have that
\begin{equation*}
\sum_{n \,=\, 1}^N \frac{L_N}{n} s_n
\end{equation*}
is an odd integer and, in particular, the sum $\sigma_N := \sum_{n \,=\, 1}^N s_n / n$ is nonzero, so that
$\m_N > 0$.
Furthermore, $|\sigma_N| \geq 1 / L_N$.
Thanks to the Prime Number Theorem, we have
\begin{equation*}
L_N = \exp(\psi(N)) = \exp(N + o(N))
\end{equation*}
as $N \to +\infty$, where $\psi$ is Chebyshev's function, and~\eqref{equ:simple} follows.
\end{proof}

\begin{pro}\label{pro:lower}
For almost all $\tau\in\R$, as $N \to +\infty$ we have
\begin{equation*}
\m_N(\tau) > \exp\!\left(-0.665N\right) .
\end{equation*}
\end{pro}
\begin{proof}
The claim follows by the Borel--Cantelli lemma: suppose we have an upper bound $\#\SS_N \leq e^{\alpha
N}$ for some $\alpha>0$, for all large enough $N$. Then for any fixed $\eps>0$
\begin{align*}
\mathcal{E} &:= \{\tau \in \R : \m_N(\tau) \leq e^{-(\alpha+\eps) N} \text{ for infinitely many } N\} \\
&= \bigcap_{M = 1}^\infty \bigcup_{N \geq M} \{\tau \in \R : \m_N(\tau) \leq e^{-(\alpha+\eps) N}\} .
\end{align*}
The Lebesgue measure of $\mathcal{E}$ is bounded by
\begin{align*}
|\mathcal{E}|
\leq \inf_{M} \sum_{N \geq M} 2e^{-(\alpha+\eps) N} \#\SS_N
\leq \inf_{M} \sum_{N \geq M} 2e^{-\eps N}
= \inf_{M} \frac{2e^{-\eps M}}{1 - e^{-\eps}}
= 0.
\end{align*}
This implies that for almost every $\tau$, the lower bound $\m_N(\tau) > e^{-(\alpha+\eps) N}$ holds for
all $N$ large enough. The upper bound for $\#\SS_N$ with $\alpha = \log 2$ is trivial, since $\#\SS_N
\leq 2^N$. The claim will follow from a slightly better estimation for this quantity. In fact, the sum
\[
\frac{s_1}{1} + \frac{s_2}{2} + \frac{s_3}{3} + \frac{s_4}{4} + \frac{s_6}{6} + \frac{s_{12}}{12}
\]
takes only $29$ different values when $s_j\in\{\pm 1\}$.
%
%
%
Thus, let
\[
F := \{\,\{k,2k,3k,4k,6k,12k\}\colon k\in D\}
\]
with
\[
D := \{2^{3a}3^{2b}m\colon m\geq 1,\ a,b\geq 0,\ 2,3\nmid m\}.
\]
With this choice for $D$ any natural number $n$ can be contained in at most one $6$-tuple. Indeed, the
numbers in $F$ associated with a given $k=2^{3a}3^{2b}m$ are
\[
\{2^{3a}3^{2b}m,    \
  2^{3a+1}3^{2b}m,  \
  2^{3a}3^{2b+1}m,  \
  2^{3a+2}3^{2b}m,  \
  2^{3a+1}3^{2b+1}m,\
  2^{3a+2}3^{2b+1}m\}
\qquad
2,3\nmid m
\]
and comparing the evaluations in $2$ and $3$ we see that no number of this family can be produced twice.
The cardinality of the union of all $5$-tuples in $F$ containing numbers $\leq N$ is $6$ times the number
of $k\in D$ which are $\leq N/12$. The number of such $k$ can be easily seen to be
\[
\frac{1+o(1)}{(1-2^{-3})(1-3^{-2})}\frac{\varphi(6)}{6}\frac{N}{12} = \Big(\frac{1}{28}+o(1)\Big)N.
\]
As said, any $6$-tuple gives rise to only $17$ different values, not $64$, thus the inequality  $\#\SS_N
\leq e^{\alpha N}$ holds for any
\[
\alpha > \Big(1-\frac{5}{28}\Big)\log 2 + \frac{1}{28}\log 29
= 0.6648\ldots,
\]
and the result follows.
\end{proof}

\section{The bounds for $\varrho$ and $\varrho_N$}\label{sec:rho}
In this section we prove Lemma~\ref{lem:rhoN}. We observe that for $0\leq x\leq \sqrt{N}$ we have
\begin{equation*}
\prod_{n=N+1}^{\infty}\cos(\pi x/n)
=\prod_{n=N+1}^{\infty}\big(1+O(( x/n)^2)\big)
=\exp(O(x^2/N))
=1+O(x^2/N),
\end{equation*}
which proves~\eqref{equ:rho2}.

We now move to the proof of~\eqref{equ:rho3}. We remark that it is sufficient to prove such inequality
for $x \in \left[N, \exp\!\left(C(\log N)^2\right)\right]$; indeed, one can reduce to this case also when
for $x<N$ since $|\varrho_{N}(x)|\leq |\varrho_{\lfloor x\rfloor}(x)|$.

For positive integers $k, N$ and for real $\delta, x \geq 0$, define
\begin{equation}\label{equ:definSk}
\mathcal{S}_k(N, \delta, x)
:= \big\{n \in \{1, \ldots, N\} : \|x / n^k\| \geq \delta \big\},
\end{equation}
where $\|y\|$ denotes the distance of $y \in \R$ from its nearest integer. By the following lemma, the
set $\mathcal{S}_1(N, \delta, x)$ plays a crucial role in the proof of~\eqref{equ:rho3}.
\begin{lem}\label{lem:varrhoNexp}
We have
\begin{equation*}
|\varrho_N(x)| \leq \exp\!\left(-\frac{\pi^2 \delta^2}{2} \cdot \#\mathcal{S}_1(N, \delta, x)\right)
\end{equation*}
for each positive integers $N$ and for all $x, \delta \geq 0$.
\end{lem}
\begin{proof}
The claim follows easily from the inequality
\begin{equation*}
|\cos(\pi x)| \leq \exp\!\left(-\frac{\pi^2\|x\|^2}{2}\right),
\end{equation*}
holding for all $x \in \R$, and from the definitions of $\varrho_N(x)$ and $\mathcal{S}_1(N,\delta,x)$.
\end{proof}
In the next two subsections we will prove a bound for $\varrho_N$ by giving two lower bounds for
$\mathcal{S}_1(N, \delta, x)$ for some suitable values of $\delta$. More precisely, in
Section~\ref{sec:arithmetic} we will complete the proof of Lemma~\ref{lem:rhoN}, showing
that~\eqref{equ:rho3} holds for all $x \in \left[N, \exp\!\left(C(\log N)^2\right)\right]$.
However, before doing this, in the next subsection we give a simpler argument proving that in the range
$x \in \left[N, \exp\!\left(C'(\log N)^2\right)\right]$ one has $|\varrho_N(x)|<1/x^2$. We remark that
this weaker inequality would still be sufficient for our application for Theorems~\ref{thm:upper}
and~\ref{thm:loclimthm}. If optimized, this argument would lead to the constant $C'=(4e)^{-2}+o(1)$.

\subsection{A short average of the number of divisors in a prescribed interval}
\label{sec:shortaverage}
In this subsection we prove the following proposition.
\begin{pro}\label{pro:secbound}
There exists $C'> 0$ such that $|\varrho_N(x)|<1/x^2$ for all sufficiently large positive integers $N$
and for all $x \in \left[N, \exp\!\left(C'(\log N)^2\right)\right]$.
\end{pro}
We start with the following lemma, which shows that the size of $\mathcal{S}_1(N, \delta, x)$ is strictly
related to the size of a certain divisor sum.
\begin{lem}\label{lem:divisoraverage}
For any $0<\delta<\frac{1}{2}$, $x\in\R$ and $N\in\N$ we have
\begin{equation*}
\frac{N}{2} - \sum_{x -\delta N \,<\, m \,<\, x + \delta N} \sum_{\substack{n \,\mid\, m \\ N / 2 \,\leq\, n \,\leq\, N}} 1
< \#\mathcal{S}_1(N, \delta, x)
< N- \sum_{x -\frac{\delta}{2} N \,<\, m \,<\, x + \frac{\delta}{2} N} \sum_{\substack{n \,\mid\, m \\ N / 2 \,\leq\, n \,\leq\, N}} 1.
\end{equation*}
\end{lem}
\begin{proof}
First we observe that
\begin{align*}
\tfrac{N}{2} - \#\big\{n \in \Z \cap [\tfrac{N}{2}, N] : \|\tfrac{x}{n}\| < \delta\big\}
<\#\mathcal{S}_1(N, \delta, x)
< N - \#\big\{n \in \Z \cap [\tfrac{N}{2}, N] : \|\tfrac{x}{n}\| < \delta\big\}.
\end{align*}
Now,
\begin{align*}
\#\big\{n \in \Z \cap [\tfrac{N}{2}, N] : \|\tfrac{x}{n}\| < \delta\big\}
&= \#\big\{n \in \Z \cap [\tfrac{N}{2}, N] : \exists \ell \in \Z \;\; \ell - \delta < x / n < \ell + \delta\big\} \nonumber \\
&= \#\big\{n \in \Z \cap [\tfrac{N}{2}, N] : \exists \ell \in \Z \;\; x - \delta n < \ell n < x + \delta n\big\} \nonumber \\
&< \#\big\{n \in \Z \cap [\tfrac{N}{2}, N] : \exists \ell \in \Z \;\; x - \delta N < \ell n < x + \delta N \big\} \nonumber \\
&= \sum_{x - \delta N \,<\, m \,<\, x + \delta N} \sum_{\substack{n \,\mid\, m \\ N / 2 \,\leq\, n \,\leq\, N}} 1
\end{align*}
and the lower bound for $\#\mathcal{S}_1(N, \delta, x) $ follows.
Similarly one obtains the upper bound.
\end{proof}
We take $\delta=\frac{4\sqrt{\log x}}{\pi} N^{-\frac{1}2}$ and assume $x\in[N,e^{N/8}]$ so that
$0<\delta<\frac{1}{2}$ and $\delta N<2x$. In particular, by Lemmas~\ref{lem:varrhoNexp}
and~\ref{lem:divisoraverage} we obtain $|\varrho_N(x)|<1/x^2$ whenever the inequality
\begin{equation}\label{equ:Dineq}
D(x,N) < N/4
\end{equation}
is satisfied, where
\begin{equation*}
D(x,N):=\sum_{x -\frac{4}{\pi} \sqrt{N \log x} \,<\, m \,<\, x + \frac{4}{\pi} \sqrt{N \log x}}
        \sum_{\substack{n \,\mid\, m \\ N / 2 \,\leq\, n \,\leq\, N}} 1.
\end{equation*}
Now, we take $w \in(0,\frac{1}{2})$ and use Rankin's trick to bound the inner sum:
\begin{align}
D(x,N)
&<   \frac{9}{\pi} \sqrt{N \log x} \cdot \max_{m\leq 2x}\sum_{\substack{n \,\mid\, m \\ N / 2 \,\leq\, n \,\leq\, N}} 1
\leq \frac{9}{\pi} \sqrt{N \log x} \cdot \max_{m\leq 2x}\sum_{\substack{n \,\mid\, m \\ N / 2 \,\leq\, n \,\leq\, N}} \left(\frac{N}{n}\right)^w\notag\\
&<   \frac{9}{\pi}  N^{\frac{1}{2}+w} \sqrt{\log x}\cdot \max_{m\leq 2x}\sigma_{-w}(m),                                             \label{equ:boundforD}
\end{align}
where, for any $s\in\R$, $\sigma_s(m)$ is defined as the sum of the $s$-th powers of the divisors of $m$.
In his lost notebook~\cite{MR1606180}, Ramanujan studied the large values of $\sigma_{-s}(n)$ for any
$s\in[0,1]$. We state his result in a slightly weaker form in the following Lemma.
\begin{lem}\label{lem:highlycomposite}
For each fixed $\varepsilon > 0$ there exists $C_1 > 0$ such that
\begin{equation*}
\sigma_{-s}(m) < \exp\!\left(C_1\frac{(\log m)^{1-s}}{\log \log m}\right) ,
\end{equation*}
for all integers $m \geq 3$ and for all $s \in [\varepsilon, 1 - \varepsilon]$.
\end{lem}
\begin{proof}
This is a consequence of \cite[(380)--(382)]{MR1606180} (see the remark before (383) on how to make the
inequalities unconditional).
See also \cite[Ch.~3, \S3, 1b]{MR2186914}.
\end{proof}
Applying the bound given in this lemma in~\eqref{equ:boundforD}, we obtain
\begin{align*}
D(x,N)
&< \frac{9}{\pi} N^{\frac{1}{2}+w} \sqrt{\log x}\cdot \exp\!\left(C_1\frac{(\log 2x)^{1-w}}{\log \log 2x}\right)
\end{align*}
for some $C_1>0$ and any $\frac{1}{4}<w<\frac{1}{2}$, $N\in\N$ and $x\in[N,e^{N/8}]$. Picking $w =
\frac{1}{2} - 1/\log\log 2x$, so that $\frac{1}{4}<w<\frac{1}{2}$ for sufficiently large $N$, this
inequality becomes
\begin{align*}
D(x,N)< \frac{9}{\pi} N\sqrt{\log x}\cdot \exp\!\left(- \frac{\log N-C_1 e\,(\log 2x)^{\frac{1}{2}}}{\log \log 2x}\right).
\end{align*}
If $x<\exp\left(C'(\log N)^2\right)$, with $C':=(2 C_1 e)^{-2}$, then this is $o(N)$ and
so~\eqref{equ:Dineq} holds for $N$ large enough. In particular, we obtain $|\varrho_N(x)|<1/x^2$ for
$x\in[N,\exp(C'(\log N)^2)]$, and the proof of Proposition~\ref{pro:secbound} is completed.

\subsection{An arithmetic construction}\label{sec:arithmetic}
Here we complete the proof of Lemma~\ref{lem:rhoN}. More specifically, we show the following proposition.
\begin{pro}\label{pro:mainbound}
For every positive $C<1/\log 4$ there exists a positive constant $E$ depending on $C$, such that
\begin{equation*}
|\varrho_N(x)| \leq \exp\!\left(- \frac{\pi^2}{400^2} \exp\!\left(E\sqrt{\log x}\right)\right)
\end{equation*}
for all $x \in [N, \exp(C(\log N)^2)]$, for all sufficiently large $N$.
\end{pro}
We start by giving a lower bound for $\#\mathcal{S}_k(N, \delta, x)$. We remind that $\mathcal{S}_k$ was
defined in~\eqref{equ:definSk}.
\begin{lem}\label{lem:SkNdxbound}
For all $a>0$, $\delta \in (0,1/2)$ and $x \in [e^{a k^2}, N^k]$, we have
\begin{equation*}
\#\mathcal{S}_k(N, \delta, x) \geq \big((1/2-\delta)(2^{-1/k}-(3/2)e^{-a}) -(2/3)^k\big)x^{1/k}
\end{equation*}
when $k$ is large enough (depending on $a$), and $N> e^{ak}$.
\end{lem}
\begin{proof}
Let $b>1$ be a parameter that will be chosen later. If $\ell$ and $n$ are integers such that
\begin{equation*}
1 \leq \ell \leq b^k-\tfrac{1}{2}
\quad\text{ and }\quad
\Big(\frac{x}{\ell + 1/2}\Big)^{1/k} < n \leq \Big(\frac{x}{\ell + \delta}\Big)^{1/k},
\end{equation*}
then it follows easily that $n \in \mathcal{S}_k(N, \delta, x)$. As a consequence,
\begin{equation}\label{eq:A1}
\#\mathcal{S}_k(N,\delta, x)
\geq \sum_{1\leq\ell\leq b^k-\tfrac{1}{2}} \Big(\frac{x}{\ell + \delta}\Big)^{1/k} - b^k.
\end{equation}
For $0 \leq s < t \leq 1$, we have the lower bounds
\begin{equation*}
\frac{1}{(1 + s)^{1/k}} - \frac{1}{(1 + t)^{1/k}}
= \frac{1}{k}\int_s^t \frac{\,\mathrm{d} y}{(1+y)^{1+1/k}}
\geq \frac{t - s}{k(1+t)^{1+1/k}}.
\end{equation*}
Applying these inequalities in~\eqref{eq:A1} with $s = \delta / \ell$ and $t = 1/(2\ell)$, we get
\begin{equation*}
\#\mathcal{S}_k(N, \delta, x)
\geq \Big(\frac{1}{2}-\delta\Big)\frac{1}{k}
      \sum_{1\leq\ell\leq b^k-\tfrac{1}{2}} \frac{x^{1/k}}{(\ell+1/2)^{1+1/k}}
    - b^k.
\end{equation*}
Since
\begin{equation*}
\sum_{1\leq\ell\leq b^k-\tfrac{1}{2}} \frac{1}{(\ell+\frac{1}{2})^{1 + 1/k}}
\geq \int_1^{b^k-\frac{1}{2}}\frac{\,\mathrm{d} y}{(y+1/2)^{1+1/k}}
= k((3/2)^{-1/k}-b^{-1})
\geq k(2^{-1/k}-b^{-1}),
\end{equation*}
this bound show that
\begin{align*}
\#\mathcal{S}_k(N, \delta, x)
\geq (1/2-\delta)\big(2^{-1/k} - b^{-1}\big)x^{1/k} - b^k.
\end{align*}
From the assumption $x \geq e^{a k^2}$ we get the claim setting $b:=2e^a/3$.
\end{proof}
Now we state a well-known identity~(see, e.g., \cite[Ch.~1, Problem~5]{MR554488}).
\begin{lem}\label{lem:identity}
For all integers $m \geq 0$, the identity
\begin{equation*}
\sum_{j \,=\, 0}^m (-1)^j \binom{m}{j} \frac{1}{x + j}
= \frac{m!}{x(x+1) \cdots (x + m)}
\end{equation*}
holds in $\mathbb{Q}(x)$.
\end{lem}
\begin{proof}
By induction on $m$.
\end{proof}
The next lemma is a simple inequality which will be useful later.
\begin{lem}\label{lem:ntokbound}
We have
\begin{equation*}
0 \leq \frac{1}{n^k} - \frac{1}{n(n + 1)\cdots(n + k - 1)}
< \frac{k^2}{2n^{k+1}},
\end{equation*}
for all positive integers $n$ and $k$.
\end{lem}
\begin{proof}
Since $1 + x \leq e^x$ for all real number $x$, we have
\begin{equation*}
0 \leq 1 - \prod_{j\,=\,0}^{k - 1} \left(1 + \frac{j}{n}\right)^{-1}
\leq 1 - \exp\!\left(-\sum_{j \,=\, 0}^{k - 1} \frac{j}{n}\right)
< 1 - e^{-k^2/(2n)}
< \frac{k^2}{2n},
\end{equation*}
and dividing everything by $n^k$ we get the desired claim.
\end{proof}
Next, using Lemmas~\ref{lem:identity} and~\ref{lem:ntokbound}, we deduce a bound for $\#\mathcal{S}_1$
from the bound for $\#\mathcal{S}_k$ given by Lemma~\ref{lem:SkNdxbound}.

Next, we use the previous lemmas to deduce a bound for $\mathcal{S}_1$ from Lemma~\ref{lem:SkNdxbound}.
\begin{lem}\label{lem:S1Ndxbound}
For all $\delta \in (0, 2^{-k}/20]$ and $x \in [4^{k^2}, N^k/(k-1)!]$, we have
\begin{equation*}
\#\mathcal{S}_1(N, \delta, x) \geq \frac{x^{1/k}}{200}
\end{equation*}
when $k$ is large enough and $N \geq k4^{k}$.
\end{lem}
The assumption $N \geq k4^{k}$ is an easy way to ensure that $4^{k^2}<N^k/(k-1)!$.
\begin{proof}
We set $\delta=d\cdot 2^{-(k+1)}$ for some $d$ that we fix later.
First, we have
\begin{align}\label{eq:A2}
\#\mathcal{S}_1(N, \delta, x)
&\geq \frac{1}{k} \cdot \#\big\{n \in \{1, \ldots, N\} \colon \exists j \in \{0, \ldots, k - 1\}\text{ with }\|x/(n + j)\| \geq \delta \big\} \\
&= \frac{1}{k} \cdot \big(N - \#\mathcal{T}_k(N, \delta, x)\big), \nonumber
\end{align}
where
\begin{equation*}
\mathcal{T}_k(N, \delta, x) := \big\{n \in \{1, \ldots, N\} \colon \|x/(n + j)\| < \delta\text{ for all } j \in \{0, \ldots, k - 1\}\big\}.
\end{equation*}
If $n \in \mathcal{T}_k(N, \delta, x)$, then for all $j \in \{0, \ldots, k - 1\}$ there exists an integer
$\ell_j$ such that
\begin{equation*}
\Big|\frac{x}{n + j} - \ell_j \Big| < \delta.
\end{equation*}
Therefore, setting
\begin{equation*}
\ell := \sum_{j=0}^{k-1} (-1)^j \binom{k-1}{j} \ell_j
\end{equation*}
and using Lemma~\ref{lem:identity}, we obtain
\begin{equation*}
\Big|\frac{x (k-1)!}{n(n+1)\cdots(n+k-1)} - \ell\Big|
\leq \sum_{j=0}^{k-1} \binom{k-1}{j} \Big|\frac{x}{n + j} - \ell_j\Big|
< 2^{k-1} \delta
= d/4.
\end{equation*}
Furthermore, assuming $n \geq \eta k (x/d)^{1/(k+1)}$ for some $\eta>0$, thanks to
Lemma~\ref{lem:SkNdxbound} we have that
\begin{align*}
\Big|\frac{x (k-1)!}{n^k} - \ell\Big|
&\leq \frac{d}{4} + \Big|\frac{x (k-1)!}{n^k} - \frac{x (k-1)!}{n(n+1)\cdots(n+k-1)}\Big|
\leq  \frac{d}{4} + \frac{x k^2(k-1)!}{2n^{k+1}}\\
&\leq \frac{d}{4} + \frac{x d k^2k^k}{2e^k\eta^{k+1}k^{k+1}x}
= \frac{d}{4}\Big(1 + \frac{2k}{\eta (\eta e)^k}\Big).
\end{align*}
Choosing $\eta> e^{-1}$ this quantity becomes $\leq 3d/10$ if $k$ is large enough (depending on the
choice of $\eta$). Choosing $d<5/3$ we ensure that this quantity is strictly smaller than $1/2$.
Therefore, under these hypotheses
\begin{equation*}
\Big\|\frac{x (k-1)!}{n^k}\Big\| \leq \frac{3d}{10}.
\end{equation*}
Summarizing, we have proved that for all $n \in \mathcal{T}_k(N, \delta, x)$, but at most $\eta
k(x/d)^{1/(k+1)}$ exceptions, it holds $n \notin \mathcal{S}_k(N, 3d/10, x(k-1)!)$. As a consequence,
\begin{equation*}
\#\mathcal{T}_k(N, \delta, x) - \eta k(x/d)^{1/(k+1)} \leq N - \#\mathcal{S}_k(N, 3d/10, x(k-1)!).
\end{equation*}
Hence, recalling~\eqref{eq:A2} and thanks to Lemma~\ref{lem:SkNdxbound}, since by hypothesis $x \geq
e^{ak^2}$ for some $a \geq 1$, we obtain
\begin{align*}
\#\mathcal{S}_1(N, \delta, x)
&\geq \frac{1}{k} \#\mathcal{S}_k(N, 3d/10, x(k-1)!) - \eta(x/d)^{1/(k+1)}\\
&\geq \Big(\Big(\frac{1}{2}-\frac{3d}{10}\Big)(2^{-1/k}-\frac{3}{2}e^{-a}) -(2/3)^k\Big)\frac{(k-1)!^{1/k}}{k} \,x^{1/k}
    - \eta(x/d)^{1/(k+1)}.
\end{align*}
Collecting $x^{1/k}$ and using the inequality $k!\geq (k/e)^k$ we get
\begin{align*}
\#\mathcal{S}_1(N, \delta, x)
&\geq x^{1/k} \Big(\Big(\Big(\frac{1}{2}-\frac{3d}{10}\Big)(2^{-1/k}-\frac{3}{2}e^{-a}) -(2/3)^k\Big)\frac{(1/k)^{1/k}}{e}
            - \frac{\eta}{d^{1/(k+1)}x^{1/(k^2 + k)}}\Big),
\end{align*}
and recalling the assumption $x\geq e^{ak^2}$, we obtain
\begin{align*}
\#\mathcal{S}_1(N, \delta, x)
&\geq x^{1/k} \Big(\Big(\Big(\frac{1}{2}-\frac{3d}{10}\Big)(2^{-1/k}-\frac{3}{2}e^{-a}) -(2/3)^k\Big)\frac{(1/k)^{1/k}}{e}
            - \frac{\eta}{d^{1/(k+1)}e^{ak/(k + 1)}}\Big).
\end{align*}
For $k$ large enough, this quantity is positive as soon as
\begin{equation*}
\Big(\frac{1}{2}-\frac{3d}{10}\Big)(1-\frac{3}{2}e^{-a}) > \frac{\eta e}{e^a}.
\end{equation*}
If $\eta$ is very close to $e^{-1}$ and $d$ is very small, this inequality is satisfied by any $a$ with
$(1-(3/2)e^{-a})>2e^{-a}$, i.e. $a > \log (7/2)$. We set $a=\log 4$, allowing the choice $\eta=0.4$
and $d=0.1$, when $k$ is large. An explicit computation shows that with these values for the parameters
the lower bound is larger than $x^{1/k}/200$ as soon as $k$ is larger than $400$.
\end{proof}
We are now ready to prove Proposition~\ref{pro:mainbound}. Let $C$ be any positive constant, $C< 1/\log
4$, and pick any $C'$ with $C<C'<1/\log 4$. We take $\delta := 2^{-k}/20$, and $k :=
\intpart{\sqrt{C'\log x}}$ for every $x$ in the given range. Then $x$ is in the interval $[4^{k^2},
N^k/(k-1)!]$. In fact, the inequality $4^{k^2} \leq x$ is evident, and
\begin{gather*}
x \leq \frac{N^k}{(k-1)!}
\Leftarrow
x \leq \Big(\frac{eN}{k}\Big)^k
\iff
\log x \leq k\log\Big(\frac{eN}{k}\Big).
\end{gather*}
Since $\sqrt{C'\log x} - 1 \leq k = \intpart{\sqrt{C'\log x}} \leq \sqrt{C'\log x}$, the last inequality
is implied by
\begin{equation*}
\frac{\log x}{\sqrt{C'\log x}-1} + \log(\sqrt{\log x}) \leq \log(eN/\sqrt{C'}).
\end{equation*}
As a function of $x$ this can be written as
\begin{equation*}
\sqrt{\frac{\log x}{C'}} + \log(\sqrt{\log x}) \leq \log N  + O_{C'}(1).
\end{equation*}
We are assuming that $\log N \leq \log x \leq C(\log N)^2$, hence this is implied by
\begin{equation*}
\sqrt{C/C'}\log N + \log\log N \leq \log N  + O_{C,C'}(1)
\end{equation*}
which is true as soon as $N$ is large enough.
This proves that we can apply Lemma~\ref{lem:S1Ndxbound}, getting
\begin{align*}
\delta^2 \cdot \#\mathcal{S}_1(N, \delta, x)
&\geq \frac{1/200}{400} \cdot 4^{-k} x^{1/k}
 =    \frac{2}{400^2} \exp\Big(\frac{\log x}{k} - k\log 4\Big) \\
&\geq \frac{2}{400^2} \exp(E\sqrt{\log x}),
\end{align*}
where $E:=\tfrac{1}{\sqrt{C'}} - \sqrt{C'}\log 4$.
Hence, applying Lemma~\ref{lem:varrhoNexp}, we get
\begin{align*}
|\varrho_N(x)|
&\leq \exp\!\Big(-\frac{\pi^2 \delta^2}{2} \cdot \#\mathcal{S}_1(N, \delta, x)\Big)
 \leq \exp\!\Big(-\frac{\pi^2}{400^2}\exp(E\sqrt{\log x})\Big),
\end{align*}
which is the claim.

\section*{Appendix}
The time needed for the computation of $\m_N$ with a direct exhaustive computation grows exponentially
with $N$ and becomes unpractical already for $N\approx 30$. Thus, for computing $\m_N$ for larger $N$ we
used the following idea. Let
\begin{equation*}
A:=\left\{\sum_{n \,=\, 1}^R \frac{s_n}{n} \colon s_1,\dots, s_R\in\{-1,+1\}\right\}
\quad
B:=\left\{\sum_{n \,=\, R+1}^N \frac{s_n}{n} \colon s_{R+1},\dots, s_N\in\{-1,+1\}\right\},
\end{equation*}
for any intermediate parameter $R\in[1, N]$. Then $\m_N = \min\{|a-b|\colon a\in A,\ b\in B\}$. The
algorithm producing this minimal distance is very fast if one preorders the lists $A$ and $B$. In this
way we were able to compute all $\m_N$ with $N\leq \np$; see the table below. The need of a large
quantity of RAM for storing the lists prevents us to compute significantly larger values of $N$. For the
computations we have used PARI/GP~\cite{PARI2}.

\begin{center}
\begin{tabular}{ rr| rr| rr| rr}
$N$&    $\m_N L_N$ & $N$&    $\m_N L_N$ & $N$&    $\m_N L_N$ & $N$&    $\m_N L_N$  \\
  \hline
 1 &             1 & 17 &            97 & 33 &        902339 & 49 &    421936433719\\
 2 &             1 & 18 &            97 & 34 &       7850449 & 50 &    175378178867\\
 3 &             1 & 19 &          3767 & 35 &       7850449 & 51 &      8643193037\\
 4 &             1 & 20 &          3767 & 36 &       7850449 & 52 &      8643193037\\
 5 &             7 & 21 &          3767 & 37 &      10683197 & 53 &    461784703049\\
 6 &             3 & 22 &          2285 & 38 &      68185267 & 54 &    461784703049\\
 7 &            11 & 23 &         24319 & 39 &      37728713 & 55 &    461784703049\\
 8 &            13 & 24 &         24319 & 40 &      37728713 & 56 &    461784703049\\
 9 &            11 & 25 &         71559 & 41 &     740674333 & 57 &    514553001783\\
10 &            11 & 26 &          4261 & 42 &     740674333 & 58 &    116096731427\\
11 &            23 & 27 &         13703 & 43 &    1774907231 & 59 &   2810673355099\\
12 &            23 & 28 &         13703 & 44 &    1774907231 & 60 &   2810673355099\\
13 &           607 & 29 &        872843 & 45 &    1774907231 & 61 &   4723651835663\\
14 &           251 & 30 &        872843 & 46 &    1699239271 & 62 & 136420009515743\\
15 &           251 & 31 &      17424097 & 47 &    3103390393 & 63 & 136420009515743\\
16 &           125 & 32 &      13828799 & 48 &    3103390393 & 64 &  23093515509397\\
\end{tabular}
\end{center}

\bibliographystyle{amsplain}

\begin{thebibliography}{10}
\bibitem{BettinMolteniSanna2}
S.~Bettin, G.~Molteni, and C.~Sanna, \emph{Greedy approximations by signed harmonic
  sums and the Thue--Morse sequence}, preprint arXiv:1805.00075,
  \url{http://arxiv.org/abs/1805.00075}, 2018.

\bibitem{MR2051473}
J.~Borwein, D.~Bailey, and R.~Girgensohn, \emph{Experimentation in
  mathematics}, A K Peters, Ltd., Natick, MA, 2004, Computational paths to
  discovery.

\bibitem{MR1341721}
D.~W. Boyd, \emph{A {$p$}-adic study of the partial sums of the harmonic
  series}, Experiment. Math. \textbf{3} (1994), no.~4, 287--302.

\bibitem{Crandall}
R.~E. Crandall, \emph{Theory of {ROOF} walks},
  \url{http://www.reed.edu/physics/faculty/crandall/papers/ROOF11.pdf}, 2008.

\bibitem{Erd32}
P.~Erd{\H{o}}s, \emph{Egy {K}\"{u}rsch\'{a}k-f\'{e}le elemi
  sz\'{a}melm\'{e}leti t\'{e}tel \'{a}ltal\'{a}nos{\'i}t\'{a}sa}, Mat. Fiz.
  Lapok \textbf{39} (1932), 17--24.

\bibitem{MR1129989}
A.~Eswarathasan and E.~Levine, \emph{{$p$}-integral harmonic sums}, Discrete
  Math. \textbf{91} (1991), no.~3, 249--257.

\bibitem{MR2076335}
R.~K. Guy, \emph{Unsolved problems in number theory}, third ed., Problem Books
  in Mathematics, Springer-Verlag, New York, 2004.

\bibitem{Morrison2}
K.~E. Morrison, \emph{Random walks with decreasing steps},
  \url{https://www.calpoly.edu/~kmorriso/Research/RandomWalks.pdf}, 1998.

\bibitem{Mor}
K.~E. Morrison, \emph{Cosine products, {F}ourier transforms, and random
  sums}, Amer. Math. Monthly \textbf{102} (1995), no.~8, 716--724. \MR{1357488}

\bibitem{Offner}
C.~D. Offner, \emph{Zeros and growth of entire functions of order {$1$}\ and
  maximal type with an application to the random signs problem}, Math. Z.
  \textbf{175} (1980), no.~3, 189--217.

\bibitem{PARI2}
The PARI~Group, Bordeaux, \emph{{PARI/GP}, version {\tt 2.6.0}}, 2013, from
  \url{http://pari.math.u-bordeaux.fr/}.

\bibitem{MR1606180}
S.~Ramanujan, \emph{Highly composite numbers}, Ramanujan J. \textbf{1} (1997),
  no.~2, 119--153, Annotated and with a foreword by Jean-Louis Nicolas and Guy
  Robin.

\bibitem{MR554488}
J.~Riordan, \emph{Combinatorial identities}, Robert E. Krieger Publishing Co.,
  Huntington, N.Y., 1979, Reprint of the 1968 original.

\bibitem{MR2186914}
J.~S\'andor, D.~S. Mitrinovi\'c, and B.~Crstici, \emph{Handbook of number
  theory. {I}}, Springer, Dordrecht, 2006, Second printing of the 1996
  original.

\bibitem{MR3486261}
C.~Sanna, \emph{On the {$p$}-adic valuation of harmonic numbers}, J. Number
  Theory \textbf{166} (2016), 41--46.

\bibitem{MR2040884}
B.~Schmuland, \emph{Random harmonic series}, Amer. Math. Monthly \textbf{110}
  (2003), no.~5, 407--416.

\bibitem{Wei02}
E.~W. Weisstein, \emph{Concise encyclopedia of mathematics}, 2 ed., CRC Press,
  2002.

\bibitem{Wol62}
J.~Wolstenholme, \emph{On certain properties of prime numbers}, Quart. J. Pure
  Appl. Math. \textbf{5} (1862), 35--39.

\bibitem{MR3608179}
B.-L. Wu and Y.-G. Chen, \emph{On certain properties of harmonic numbers}, J.
  Number Theory \textbf{175} (2017), 66--86.
\end{thebibliography}

\end{document}